\documentclass{article}
\usepackage[T1]{fontenc} % check whether it works
\usepackage{relsize}
\usepackage[utf8]{inputenc}
\usepackage{realboxes}
\usepackage{dsfont}
\usepackage{soul}
\usepackage{marvosym}
\usepackage[normalem]{ulem}
\usepackage{leftindex}
\usepackage{float}
\usepackage{dashbox}
\usepackage{fancybox}
\usepackage[english]{babel}
\usepackage{graphicx}%
\usepackage{multirow}%
\usepackage{amsmath,amssymb,amsfonts}%
\usepackage{amsthm}%
\usepackage{mathrsfs}%
\usepackage[title]{appendix}%
\usepackage[svgnames]{xcolor}%
\usepackage{textcomp}%
\usepackage{manyfoot}%
\usepackage{booktabs}%
\usepackage{algorithm}%
\usepackage{algorithmicx}%
\usepackage{algpseudocode}%
\usepackage{listings}%
\usepackage{tikz-cd}
\usepackage{pgfplots}
\pgfplotsset{compat=1.13}
\usepackage[all,cmtip]{xy}
\usepackage{graphicx,tikz-cd,pgf}
\usetikzlibrary{positioning, shapes}
\usetikzlibrary{arrows.meta}
\usepackage{caption}
\usepackage{tkz-euclide}

\usepackage{stmaryrd} \usepackage{trimclip}
\makeatletter
\DeclareRobustCommand{\shortto}{\mathrel{\mathpalette\short@to\relax}}
\newcommand{\short@to}[2]{
  \mkern2mu
  \clipbox{{.5\width} 0 0 0}{$\m@th#1\vphantom{+}{\shortrightarrow}$}}
\makeatother

\makeatletter
\DeclareRobustCommand{\lshortto}{\mathrel{\mathpalette\lshort@to\relax}}
\newcommand{\lshort@to}[2]{
  \mkern2mu
  \clipbox{{.1\width} 0 0 0}{$\m@th#1\vphantom{+}{\shortleftarrow}$}}
\makeatother

\newcommand{\cfunction}[1]{{\mathsf{1}}_{#1}}

\DeclareMathOperator{\PoissonISymbol}{\mathsf{P}}
\newcommand{\PoissonI}[1]{\PoissonISymbol[#1]}

\newcommand{\newradius}{s}

\DeclareMathOperator{\oscSymbol}{osc}
\newcommand{\osc}[3]{\oscSymbol_{#3}(#1,#2)}

\DeclareMathOperator{\nullsets}{\boldsymbol{\mathcal{N}}}

\DeclareMathOperator{\almostev}{\text{\tiny{a.e.}}}

\newcommand{\car}{\bafNAME}
\newcommand{\bafNAME}{\boldsymbol{\mathsf{A}}}

\newcommand{\baffz}{% Firm, Zero variables
\bafNAME_{\fbsf}
}

\newcommand{\bafaz}{% Attached, Zero variables
\bafNAME_{\absf}
}

\newcommand{\bafcz}{% curvilinear, Zero variables
\bafNAME_{\cbsf}
}

\newcommand{\bafsz}{% sequential, Zero variables
\bafNAME_{\sbsf}
}

\newcommand{\baftNAME}{\bafNAME_{\tbsf}}

\newcommand{\baftz}{% Zero variables
\baftNAME
}

\newcommand{\aesubset}[2]{
#1
\overset{
\scriptstyle{\almostev}
}
{\subset}
#2
}

\newcommand{\bhd}{\bdryhd}
\newcommand{\bdryhd}{\partial\hd}

\newcommand{\aeequal}[2]{
#1
\overset{
\scriptstyle{\almostev}
}
{=}
#2
}

\newcommand{\sobb}{
\mathcal{B}
}

\newcommand{\bb}{\beta}

\newcommand{\acube}[2]{% cube (with) center (and) arc-length
\mathsf{cube}(#1,#2)
}

\newcommand{\acal}[2]{% arc (with) center (and) arc-length
\beta(#1,#2)
}

\newcommand{\Qb}{\bchar{Q}}

\newcommand{\Nb}{\bchar{N}}

\newcommand{\Xb}{\bchar{X}}

\newcommand{\Bb}{\bchar{B}}
\newcommand{\Vb}{\bchar{V}}
\newcommand{\wVb}{\widetilde{\Vb}}
\newcommand{\Wb}{\bchar{W}}
\newcommand{\Cb}{\bchar{C}}

\newcommand{\Ab}{\bchar{A}}

\newcommand{\Ob}{\bchar{O}}
\newcommand{\Sb}{\bchar{S}}
\newcommand{\Ub}{\bchar{U}}

\newcommand{\vb}{\bchar{v}}
\newcommand{\Zb}{\bchar{Z}}
\DeclareMathOperator{\vob}{\vb_{o}}

\newcommand{\lingvalue}[3]{% 1 function 2 approach region 3 boundary point
\lim_{\stackrel{z\in #2}{z\to #3}}{#1(z)}
}

\definecolor{mygray}{rgb}{0.8,0.8,0.8}
\makeatletter
    \newcommand{\colorboxed}[3][white]{\fcolorbox{#2}{#1}{\m@th$\displaystyle#3$}}
\makeatother

\DeclareMathOperator{\dpoint}{\dchar{z}}

\DeclareMathOperator{\bpoint}{\mathsf{w}}
\DeclareMathOperator{\bopoint}{\bpoint_{o}}

\newcommand{\dnewsegnato}{{q}_{o}}
\newcommand{\bsegnato}{{b}^{\prime}}
\newcommand{\zsegnato}{\tilde{\dpoint}}
\DeclareMathOperator{\bpointv}{\mathsf{v}}
\DeclareMathOperator{\bpointx}{\mathsf{x}}

\newcommand{\ballName}{{{B}}}
\newcommand{\ball}[2]{\ballName(#1,#2)}

\newcommand{\ballrel}[2]{\approachr_{#1}^{#2}}

\newcommand{\dimn}{d}
\newcommand{\dprime}{\dimn'}

\DeclareMathOperator{\ntbvSymbol}{
\boldsymbol{\flat}
}
\newcommand{\ntbv}[1]{
{#1}_{\ntbvSymbol}
}

\newcommand{\angolo}{b}

\makeatletter
\newcommand*\bigcdot{\mathpalette\bigcdot@{.5}}
\newcommand*\bigcdot@[2]{\mathbin{\vcenter{\hbox{\scalebox{#2}{$\m@th#1\bullet$}}}}}
\makeatother

\DeclareMathSymbol{\mhyphen}{\mathord}{AMSa}{"39}

\DeclareMathOperator{\udone}{%BoundaryOfTheUnitDisc
\mathbb{D}
}

\usepackage[mathscr]{euscript}

\usepackage{stmaryrd} \usepackage{trimclip}

\newcommand{\absf}{\boldsymbol{\mathsf{a}}}

\newcommand{\cbsf}{\boldsymbol{\mathsf{c}}}

\newcommand{\fbsf}{\boldsymbol{\mathsf{f}}}

\newcommand{\sbsf}{\boldsymbol{\mathsf{s}}}

\newcommand{\tbsf}{\boldsymbol{\mathsf{t}}}

\newcommand{\dfunctionharm}{\mathtt{u}}

\newcommand{\dfunction}{\mathtt{f}}
\newcommand{\bfunction}{\mathsf{f}}
\newcommand{\bfun}{\bfunction}
\newcommand{\bfung}{\mathsf{g}}
\newcommand{\bfunp}{\mathsf{p}}

\DeclareMathOperator{\approachr}{% approach r(egion)
\af}

\DeclareMathOperator{\af}{\Lambda}

\newcommand{\botud}{\partial\udone}

\DeclareMathOperator{\tpsSymbol}{\mathcal{P}}
\newcommand{\tps}[1]{
\tpsSymbol{\kern-1.1pt}{(#1)}}

\newcommand{\dchar}[1]{\mathtt{#1}}
\newcommand{\bchar}[1]{\mathsf{#1}}

\DeclareMathOperator{\FatousetSymbol}{\mathsf{Fatou}}
\newcommand{\Fatouset}[1]{ % 1 function 
\FatousetSymbol[{#1}]
}

\DeclareMathOperator{\hinfty}{\hSymbol^{\infty}}

\newcommand{\hd}{E_{\dimn}^{+}}

\DeclareMathOperator{\hSymbol}{\mathtt{h}}

\newcommand{\distanza}[3]{%1=first set, 2=second set, 3=radius
\dist\left[
A_1\cap\partial B(w,r),A_1\cap\partial B(w,r)
\right]
}

\newcommand{\eqdef}{\overset{\mathrm{def}}{=\joinrel=}}

\DeclareMathOperator{\NN}{{\mathbb{N}}}

\DeclareMathOperator{\RR}{\mathbb{R}}
\DeclareMathOperator{\dist}{dist}

\DeclareMathOperator{\CC}{%BoundaryOfTheUnitDisc
\mathbb{C}
}

\newtheoremstyle{slthmstyle}  % name of the style to be used
   {}       % measure of space to leave above the theorem. E.g.: 3pt
   {}       % measure of space to leave below the theorem. E.g.: 3pt
   {\slshape}   % name of font to use in the body of the theorem
   {}        % measure of space to indent
   {\bfseries}  % name of head font
   {}   % punctuation between head and body
   {2mm}       % space after theorem head
   {}           % Manually specify head
\theoremstyle{slthmstyle}
\newtheorem{theorem}{Theorem}[section]
\newtheorem*{corollary*}{Corollary.}

\newtheorem{lemma}[theorem]{Lemma}

\theoremstyle{definition}

\newtheorem{claim}[theorem]{Claim}
\theoremstyle{definition}

\newtheorem*{Littlewood}{Littlewood's Theorem (1927)}

\newtheorem*{Rudin'sTheorem}{Rudin's Theorem (1979)}
\newtheorem*{Rudin'sVOIFunction}{A Very Oscillatory Inner Function (1979)}

\numberwithin{equation}{section}

\title{A Littlewood-Type Theorem for Harmonic Functions in Euclidean Half-Spaces}
\author{Fausto Di Biase  \\
Dipartimento di Economia\\
Universit\`a ``G.\! D'Annunzio'' di Chieti-Pescara,\\
Viale Pindaro, 42, I-65127, Pescara, Italy\\
email: {\tt fausto.dibiase@unich.it}\\
\\	Haguma Gratien \\
Department of Mathematics\\ University of Rwanda\\
KN 67 Nyarugenge, Kigali, 3900 Kigali, Rwanda\\
email: {\tt hagugrat@gmail.com}\\
\\	Olof Svensson \\
Department of Science and Technology\\ Link\"oping University\\
SE-60174 Norrk\"oping,  Sweden\\
email: {\tt olof.svensson@liu.se}
}

\date{\today}
\begin{document}

\maketitle

\begin{abstract}

In 1927 J.E. Littlewood proved that, for bounded harmonic functions on the unit disc, the Fatou-type result, on the existence of almost everywhere boundary values through  approach regions that are nontangential, will fail for any system of \textit{tangential} approach regions that has the following two additional properties: 
It consists of \textit{curves} ending at the various boundary points, and it is rotationally invariant. 

However, in 1984 A. Nagel and E.M. Stein — elaborating results of Rudin (1979) and Nagel, Rudin and Shapiro (1982) — proved the existence of rotationally-invariant systems of tangential \textit{sequences} along which a Fatou-type result holds, i.e., along which every bounded harmonic function converges a.e. to its nontangential boundary values. 
Moreover, they extended their result to higher-dimensional Euclidean 
half-spaces. 

The Nagel-Stein result has prompted 
the question of formulating and proving a Littlewood type theorem that can also can be applied to tangential approach regions which are \textit{sequential}. 

In this paper we prove a Littlewood-type theorem for bounded harmonic functions in higher-dimensional Euclidean half-spaces, for systems of tangential approach regions which may very well be \textit{sequential}.

This is the first result of this kind, apart from a recent result of ours whose setting is the unit disc. Indeed, in all the other results of Littlewood type, the tangential approach regions were required to be \textit{curvilinear} or at least to possess a certain topological property that 
\textit{excluded} the possibility that they could be sequential.

\end{abstract}

{\footnotesize

MSC 31B25, 42B30.}

\section{Introduction}

In Euclidean settings, the boundary behavior of bounded harmonic functions, 
in the sense of the existence of almost everywhere boundary values along nontangential approach regions,  
has been studied in the following three contexts: 

\noindent (i) the unit disc $\udone\eqdef\{z\in\CC:
|z|<1\}$ in the complex plane, where $|z|$ is Euclidean length in $\CC$ 
(\cite{Fatou1906}, 
\cite{Privalov1956}, 
\cite{LohwaterPiranian1957}, 
\cite{CollingwoodLohwater1966}, 
\cite{Duren}, 
\cite{Hoffman},  
\cite{Koosis}, 
\cite{Pommerenke}, 
\cite{DiBiaseStokolosSvenssonWeiss2006}, \cite{Zygmund1949}); 

\noindent (ii)
in Euclidean half-spaces $\hd$ in $\RR^n$, defined by  
$$
\hd\eqdef\{x\in\RR^{\dimn}:x(\dimn)>0\}\subset\RR^\dimn
$$ 
where $\dimn\geq2$ and $x\equiv(x(1),x(2),\ldots,x(\dimn))$, for $x\in\RR^{\dimn}$
(\cite{Calderon1950}, \cite{Stein1961}, 
\cite{Stein1970bis},
\cite{SteinWeiss1971},
\cite{Fefferman-Stein1972}, \cite{Zygmund1988}, 
\cite{Stein1993}); 

(iii) in NTA domains in $\RR^n$ (a class of bounded domains which properly includes Lipschitz domains — \cite{JerisonKenig1982}, \cite{DiBiase1998}). 

Results on the 
existence of almost everywhere boundary values along nontangential approach regions 
are known as theorems \textit{of Fatou type}: They have been extended to more general settings, for example, 
Riemannian manifolds of pinched negative curvature (\cite{AndersonSchoen1985},
\cite{Ancona1987}); for related results, see \cite{ArcozziDiBiaseUrbanke1996},
\cite{CifuentesDorronsoroSueiro1992},
\cite{DiBiase2009},
\cite{DiBiaseWeiss2010},
\cite{DiBiaseWeiss2010 bis},
\cite{DiBiaseKrantz2021bis},
\cite{DiBiaseKrantz2023},
\cite{DiBiaseKrantz2024},
\cite{DiBiaseKrantz2025},
\cite{HakimSibony1983},
\cite{Hirata2005},
\cite{Koranyi1969abis},
\cite{Krantz1991}, 
\cite{Krantz2001}, 
\cite{Krantz2007},  
\cite{Krantz2019},
\cite{Sueiro1986},
\cite{Sueiro1987},
\cite{Sueiro1990}.

Theorems \textit{of Littlewood type} assert the \textit{non-existence} 
of almost everywhere boundary values along \textit{tangential} approach regions.  
Here the central notion  is that of a \textit{system of approach regions}. 
In the setting of the unit disc, a \textit{system of approach regions in $\udone$} is 
a 
family 
$\af\equiv{\{\af(\bpoint)\}}_{\bpoint\in\botud}$ (where 
$\botud$ is the boundary of the unit disc) such that
$\af(\bpoint)\subset\udone$ and 
$\bpoint 
\text{belongs to the closure of }
\af(\bpoint)
\text{ in }
\CC$, $\forall \bpoint\in\botud$. 
J.E. Littlewood was the first to prove a result of this kind in $\udone$ (\cite{Littlewood1927}). 
\begin{Littlewood}
Assume that $\af$ is a system of approach regions in $\udone$ and that 
\begin{description}
\item[(1)] $\af(\bpoint)$ is tangential to $\botud$ at $\bpoint$, 
for each $\bpoint\in\botud$; 
\item[(2)] $\af(\bpoint)$ is a curve in $\udone$ 
ending at $\bpoint$, for each $\bpoint\in\botud$;
\item[(3)] $\af$ is rotationally invariant.
\end{description}
Then there exists a bounded harmonic function $\dfunction$ on $\udone$ such that  
\begin{equation}
\text{for a.e.}\bpoint\in\botud, 
\lingvalue{\dfunction}{\af(\bpoint)}{\bpoint}
\text{ does not exist}
\label{e:lvdneudone} 
\end{equation}
\end{Littlewood} 

Littlewood's Theorem has been 
\textit{extended} to Euclidean half-spaces (\cite{Aikawa1991}), 
but not yet to NTA domains in 
$\RR^n$.

Littlewood's Theorem has been \textit{refined} in various ways 
(\cite{Aikawa1990},
\cite{Aikawa1991},
\cite{DiBiaseStokolosSvenssonWeiss2006}, 
\cite{DiBiaseGratienSvensson}). 

One refinement that has been the object of recent investigations pertains \textbf{(3)}. 
Indeed, it has been understood that it is not necessary to require that $\af$ is rotationally invariant, but it suffices to ask that it is \textit{regular}, i.e., that for each open set 
$U\subset\udone$ the set $\{\bpoint\in\botud:\af(\bpoint)\cap{}U\not=\emptyset\}$ is  measurable. 

Another refinement that has been the object of recent investigations pertains \textbf{(2)}:  
The goal here is to understand the role played by this hypothesis, which 
is particularly strong, since it 
\textit{excludes} that the approach regions may end \textit{sequentially} at the boundary. 
In the present paper, we contribute to this problem in the setting of Euclidean half-spaces. 
A first step in this direction has been achieved in~\cite{DiBiaseStokolosSvenssonWeiss2006}, 
were it was shown that it is not necessary to ask 
that each approach region $\af(\bpoint)\subset\udone$ is a curve, 
but it suffices to ask that it is \textit{attached to} $\bpoint$, i.e., the set $\{\bpoint\}\cup\af(\bpoint)\subset\CC$ is 
connected, $\forall\bpoint\in\botud$; observe that this topological property still 
\textit{excludes} the possibility that 
an approach region of this kind may be a \textit{sequence} 
ending at $\bpoint$.  A second step was achieved in 
\cite{DiBiaseGratienSvensson}, in the context of the unit disc, 
where instead of \textbf{(2)} a different hypothesis 
is assumed, which allows the approach regions to end sequentially to the boundary. 

\subsection{The significance of our main result}

Until recently, all the theorems of Littlewood type given so far are {restricted} to 
systems of tangential approach regions which are either \textit{curvilinear} or 
share with the curvilinear approach regions a topological property 
that \textit{excludes} the possibility that they may end \textit{sequentially} at the given boundary points. 

The first step toward the removal of this restriction has been achieved in 
\cite{DiBiaseGratienSvensson}, in the setting of the unit disc. 
In this paper, 
we take another step in this direction, and 
 remove the aforementioned restriction in the setting of Euclidean half-spaces. 
Indeed, our result \textit{also} applies to tangential approach regions that are \textit{sequential}. Our result can be better appreciated if we recall that 
A. Nagel and E.M. Stein \cite{Nagel--Stein1984}, elaborating a result of 
W. Rudin \cite{Rudin1979bis} and A. Nagel, W. Rudin and J.H. Shapiro \cite{NagelRudinShapiro}, proved 
the existence 
of translation invariant systems of tangential 
\textit{and sequential}
approach regions 
in $\hd$
along which \textit{all} bounded harmonic functions on $\hd$ 
\textit{converge a.e.} to 
their nontangential boundary values. See also \cite{Rudin1988}.

In order to achieve our result, we have identified a new property, 
that approach regions may have, called \textit{firmness} — see~\eqref{e:firmONE} —  
which does not depend on the \textit{continuous} or \textit{discrete} nature of the 
approach regions involved, and hence 
is \textit{transversal}, so to speak, 
to the properties of being \textit{sequential} or  \textit{curvilinear}. 

In our new theorem of Littlewood type in Euclidean half-spaces 
we use the following notation. 
The boundary of $\hd$ in $\RR^{\dimn}$ is denoted by 
$\bdryhd$. Then $\bhd=\{(z,0):z\in\RR^{\dprime}\}$,  where
$\dprime\eqdef\dimn-1$. 
The Euclidean ball in $\RR^{\dimn}$ of 
center $x\in\RR^{\dimn}$ and radius $r>0$ is denoted by 
$\ball{x}{r}$, hence $\ball{x}{r}\eqdef\{y\in\RR^{\dimn}:|y-x|<r\}$,
where 
$|z|$ is the Euclidean length of $z\in\RR^{\dimn}$.

An \textit{approach region in $\hd$ at $\bpoint\in\bhd$} is a set $A\subset\hd$ such that, 
for each $r>0$, the set $A\cap\ball{\bpoint}{r}$, 
called \textit{the $r$-tail of $A$ at $\bpoint$}, is nonempty. 

If $A$ is an approach region in $\hd$ at $\bpoint\in\bhd$, and 
$\angolo\in\NN$, we say that \textit{$A$ is $\angolo$-firm at $\bpoint$} if 
\begin{equation}
\begin{aligned}
\forall r>0\,
\,
\exists \rho > 0\,
\,
\forall \bpointv & \in \bdryhd\cap\ball{\bpoint}{\rho}
\,\,
\exists y \in A\cap \ball{\bpoint}{r}
\text{ such that}
\\
& 
 |\bpointv-y| < (1+\angolo) y(\dimn)
\end{aligned}
\label{e:firmONE}
\end{equation}
We will see momentarily the geometric meaning of~\eqref{e:firmONE}.

A
\textit{system of approach regions in $\hd$} is a 
family 
$\af\equiv{\{\af(\bpoint)\}}_{\bpoint\in\bdryhd}$ such that
$\af(\bpoint)$ is an approach region in $\hd$ at $\bpoint$, 
$\forall \bpoint\in\bdryhd$. 

The collection of all systems of approach regions in $\hd$ is denoted by $\car$.
If $\af\in\car$, then
\begin{enumerate}
\item For $\bpoint\in\bhd$ and $r>0$, 
the $r$-tail of $\af(\bpoint)$ at $\bpoint$ 
is denoted by $\ballrel{\bpoint}{r}$, and hence 
$\ballrel{\bpoint}{r}\eqdef \af(\bpoint)\cap \ball{\bpoint}{r}$;

\item 
If 
$\displaystyle{\inf_{r>0} \sup_{x \in \ballrel{\bpoint}{r}} \frac{x(\dimn)}{|x-\bpoint|}=0}$, 
$\forall\bpoint\in\bdryhd$, 
then 
$\af$ is called 
a system of \textit{tangential} approach regions in $\hd$;
$\baftz\subset\car$ denotes the set of all 
systems of \textit{tangential} approach regions in $\hd$.

\item 
If $\af\in\car$ and, 
for each open $U \subset \hd $, the set 
$ \left\{ \bpoint\in\bdryhd: U \cap \af(\bpoint) \neq \emptyset  \right\}$
 is a measurable subset of $\bhd$,
then 
$\af$ is said to be \textit{regular}. 

\end{enumerate}

\subsection{Our main result}

\begin{theorem}
Assume that 
\begin{description}
\item[(1)] 
$\af$ is a  system of \textup{tangential} approach regions in $\hd$;
\item[(2)] for each $\bpoint\in\bhd$ there exists 
$\angolo\in\NN$ such that $\af(\bpoint)$ is $\angolo$-firm at $\bpoint$;
\item[(3)] $\af$ is regular.
\end{description}
Then there exists a bounded harmonic function $\dfunction$ on $\hd$ such that  
\begin{equation}
\text{for a.e.}\bpoint\in\bdryhd, 
\lingvalue{\dfunction}{\af(\bpoint)}{\bpoint}
\text{ does not exist}
\label{e:lvdne} 
\end{equation}
\label{thm:maintheorem}
\end{theorem}
We emphasize the fact that, while the previously known theorem of Littlewood type 
apply to systems of tangential approach regions that are \textit{curvilinear} \cite{Littlewood1927} and 
\cite{Aikawa1990} (for $n=2$), \cite{Aikawa1991} (for $n\geq2$), \cite{Zygmund1949} (for $n=2$), 
or share with them a topological property that \textit{excludes} the possibility 
that they may be \textit{sequential} 
\cite{DiBiaseStokolosSvenssonWeiss2006} (for $n=2$), 
our result also applies to tangential approach regions that are \textit{sequential}. 
The following notation will be useful. 

\noindent
$\baffz$ is the set of $\af\in\car$ that satisfy the condition (2) 
in Theorem~\ref{thm:maintheorem}.

\noindent 
$\bafcz$ is the set of all 
systems of \textit{curvilinear} approach regions in $\hd$, i.e., such that 
$\forall$ $\bpoint\in\bhd$ 
there exists 
a continuous function $\varphi_{\bpoint}:[0,1)\to\hd$ and 
$r_{\bpoint}>0$ such that 
$\displaystyle{\lim_{s\uparrow1}\varphi_{\bpoint}(s)=\bpoint}$ 
and 
${\ballrel{\bpoint}{r_{\bpoint}}=\ball{\bpoint}{r_{\bpoint}}
 \cap\{\varphi_{\bpoint}(s):0\leq s<1\}}$.  

\noindent 
$\bafsz$ is the set of all 
systems of \textit{sequential} approach regions in $\hd$, i.e., such that 
such that 
$\forall$ 
$\bpoint\in\bhd$ 
$\exists$ 
$r_{\bpoint}>0$ 
and there exists a sequence 
$\varphi_{\bpoint}:\NN\to\hd$
such that 
$\displaystyle{\lim_{j\to+\infty}\varphi_{\bpoint}(j)=\bpoint}$
and
$\displaystyle{\ballrel{\bpoint}{r_{\bpoint}}=\ball{\bpoint}{r_{\bpoint}}
 \cap\{\varphi_{\bpoint}(j):j\in\NN\}}$.

\noindent Then 
\begin{equation}
\text{(i)}\quad\baffz\cap\bafcz\not=\emptyset;\quad \text{(ii)}
\quad\baffz\cap\bafsz\not=\emptyset;
\label{e:good}
\end{equation}
and
$$
\bafsz\cap\baftz\cap\baffz\not=\emptyset\text{ and }
(\bafsz\cap\baftz)\setminus\baffz\not=\emptyset
$$
The property of \textit{firmness},  
introduced in the present work, should be compared 
to the notion 
of being \textit{attached}, 
introduced in \cite[p. 48]{DiBiaseStokolosSvenssonWeiss2006} for $n=2$ and therein denoted (c$\star$). If we denote by $\bafaz$ the collection of all systems of 
\textit{attached} approach regions in $\hd$, then  
\begin{equation}
\text{(i)}\quad\bafcz\subsetneq\bafaz\quad\text{ but }\quad \text{(ii)}\quad\bafaz\cap\bafsz=\emptyset 
\label{e:attachednew} 
\end{equation}
It is useful to compare \eqref{e:attachednew} (ii) with~\eqref{e:good} (ii).

\subsection{Further background results}

We now briefly give other background results that also help us assess the new contribution of this work and put it in context. 
 
In 1949, A. Zygmund \cite{Zygmund1949} gave a real-variable proof of Littlewood's theorem, that may be extended to $\hd$ so as to yield the following result: If 
$\af$ is translation invariant system of approach regions and, for each $\bpoint\in\bdryhd$, 
$\af(\bpoint)$ is a hypersurface of dimension $n-1$ tangential to $\bdryhd$, then 
there exists a bounded harmonic function $\dfunction$ on $\hd$ 
such that~\eqref{e:lvdne} holds.

In 1991, H. Aikawa \cite{Aikawa1991} proved that if $\af\in\baftz\cap\bafcz$ and  
the curves ${\{\af(\bpoint)\}}_{\bpoint\in\bdryhd}$ are uniformly bi-Lipschitz equivalent,  then there exists a bounded harmonic function $\dfunction$ on $\hd$ 
such that~\eqref{e:lvdne} holds.

In 2006, the following statement 
was proved to be independent of ZFC \cite{DiBiaseStokolosSvenssonWeiss2006}: 
\begin{quote}
\flqq There is no $\af\in\baftz\cap\bafcz$ such that each bounded harmonic function on 
$E_{2}^{+}$ converges along $\af$ to its nontangential boundary values.\frqq
\end{quote}
The independence of the statement quoted above means that it is neither possible to prove it, nor to disprove it, within ZFC. 

The notions and the results of logic that are employed in the proof of this 
result may be found in \cite{Bourbaki.Insiemi},
\cite{Bourbaki},
\cite{Bourbaki2007}, \cite{Drake1974}, \cite{Jech1978},
\cite{Kunen1980}, \cite{TomaszWeiss2020}.

In the context of NTA domains, a theorem of Fatou type has been proved in 
\cite{JerisonKenig1982}, 
and a theorem of Nagel-Stein type has been proved in \cite{DiBiase1998}. 
There is currently no result of Littlewood type for NTA domains; The approach outlined in this paper should yield such a result, and we hope to return to this matter 
in the near future. 

\subsection{The geometric meaning of the notion of firmness}

The following notions will illustrate the geometric meaning of the notion of firmness.
If $X$ is a set then $\tps{X}$ is the collection of all subsets of $X$. 

If $\af$ is a system of approach regions in $\hd$ 
and $U\subset\hd$ then the set 
$\left\{\bpoint\in\bdryhd: U \cap \af(\bpoint) \neq \emptyset  \right\}$ is denoted by 
$\af^{\ast}(U)$ and called the $\af$-shadow of $U$. 
The map  $\af^{\ast}:\tps{\hd}\to\tps{\bdryhd}$ thus defined is called 
the \textit{projection} associated to $\af$.
Hence a system of approach regions $\af$ in $\hd$ is regular if 
the projection associated to $\af$ maps each open subset of $\hd$ to a measurable 
subset of $\bhd$. 

If $\angolo\in\NN$ and $\bpointv\in\bdryhd$ 
then $\Gamma_{\angolo}(\bpointv)\eqdef\{y\in\hd: |\bpointv-y| < (1+\angolo) y(\dimn)\}$;
$\Gamma_{\angolo}(\bpointv)$ is called the 
\textit{nontangential approach region at} $\bpointv$ of \textit{aperture} $\angolo$; 
cf. \cite{Stolz1875} and \cite{SteinWeiss1971}. Observe that $\Gamma_{\angolo}$ is a system of approach regions in $\hd$.  

If $\bpoint\in\bhd$ and $r>0$ the set 
$\acal{\bpoint}{r}\eqdef\bhd\cap\ball{\bpoint}{r}$ is called 
a \textit{boundary ball in} $\RR^{\dimn}$ of center $\bpoint$. 
The collection of all boundary balls in $\hd$ is denoted $\sobb$. 
Hence~(2) in Theorem~\ref{thm:maintheorem} means that $\exists\,\angolo\in\NN$ such that $\forall r>0$ 
the $\Gamma_b$-shadow of $\ballrel{\bpoint}{r}$ contains 
a boundary ball of center $\bpoint$.

\section{Structure of the proof  of Theorem~\ref{thm:maintheorem}}
\label{s:proofofthemaintheorem}

Throughout this section, we assume, without further notice, 
that $\af$ satisfies the hypotheses of Theorem~\ref{thm:maintheorem}.

The following notions or notation will be used in the proof of 
Theorem~\ref{thm:maintheorem}, presented in 
Section~\ref{s:proofofmainresult}.  
 If $\dfunction:\hd\to\RR$, the \textbf{Fatou set of} $\dfunction$ is defined as 

$$
\Fatouset{\dfunction}
\eqdef
\{\bpoint\in\bhd: 
\exists\ell\in\RR
\text{ such that }
\lim_{\stackrel{\text{nt}}{z\to\bpoint}}\dfunction(z)=\ell\},
$$ 
where the notation 
${\displaystyle{\lim_{\stackrel{\text{nt}}{z\to\bpoint}}\dfunction(z)=\ell}}$ means 
that $\forall\angolo\in\NN$ and $\forall\epsilon>0$ 
$\exists\delta>0$ such that if $z\in\ball{\bpoint}{\delta}\cap\Gamma_{\angolo}(\bpoint)$ 
then $|\dfunction(z)-\ell|<\epsilon$. 

Denote by $\hinfty$ the space of real-valued and 
harmonic functions 
defined on $\hd$ and bounded therein. 
If $\dfunction\in\hinfty$ then $\Fatouset{\dfunction}$
has 
full measure in $\bhd$ (see \cite{Zygmund1988} and references therein), 
and the  
\textbf{nontangential boundary function of} $\dfunction$ is 
the function 
$\ntbv{\dfunction}:\Fatouset{\dfunction}\to\RR$, defined by 
$$
\ntbv{\dfunction}(\bpoint)\eqdef
\lim_{\stackrel{\text{nt}}{z\to\bpoint}}\dfunction(z). 
$$
%\section{Proof of Theorem~\ref{t:mainresult1}}\label{s:Proof}
%The following constructions will be used in the proof.
If $\dfunctionharm \in {\hinfty}$ and $\bpoint \in \bdryhd$, we define
$$
\osc{\dfunctionharm}{\bpoint}{\af}\eqdef 
\limsup_{\stackrel{\dpoint\to\bpoint}{\dpoint\in\af(\bpoint)}} 
\dfunctionharm(\dpoint)
-
\liminf_{\stackrel{\dpoint\to\bpoint}{\dpoint\in\af(\bpoint)}} 
\dfunctionharm(\dpoint)
$$
%In order to prove Theorem~\ref{t:mainresult1}, it suffices to prove 
%the following statement
%\begin{equation}
%\text{there exists }
%\dfunctionharm\in{\hinfty}
%\text{ with } 
%\dfunctionharm>0
%\text{ and } 
%\osc{\dfunctionharm}{\bpoint}>0
%\text{ for a.e. } 
%\bpoint\in\botud
%\end{equation}
%Indeed, 
%if $\dfunctionharmc$ is the harmonic conjugate to $\dfunctionharm$ then 
%$\dfunctionhol\eqdef{}e^{-\dfunctionharm-i\dfunctionharmc}$ will have the required properties. 

\subsection{Carleson tents and the Zygmund map}

If $\bb=\acal{\bpoint}{r}\subset\bhd$ is a boundary ball, 
the \textbf{Carleson tent above} $\bb$
is defined by 
\begin{equation}
\boldsymbol{\Delta}(\bb)\eqdef
\hd\cap\ball{\bpoint}{r}\subset\hd.
\label{e:Delta} 
\end{equation}
%\begin{equation}
%\boldsymbol{\Delta}(\acal{\bpoint}{\theta})\eqdef
%\hd\cap\ball{\bpoint}{|\yb-\yb{}e^{i\frac{\theta}{2}}|}.
%\label{e:Delta} 
%\end{equation}
The \textbf{Zygmund map} is the function
$\Zb:\tps{\bdryhd}\to\tps{\bdryhd}$
defined as follows: If $\Vb\subset\bdryhd$ then 
$\Zb(\Vb)\eqdef\{\bpoint\in\bdryhd: \text{\textup{{(i)}}}\,\&\,\text{\textup{{(ii)}}}\text{ hold}\}$, where 
\begin{equation} 
\begin{aligned}
&\text{\textup{(i)}}\,\,  \bpoint\in\bdryhd\setminus\Vb;
\\
&\text{\textup{(ii)}}\,\,
\forall\epsilon>0
\,\,
\exists 
\, \bb\in\sobb:
\bb\subset\Vb\cap\ball{\bpoint}{\epsilon}
\,\&\,
\af(\bpoint)\cap\boldsymbol{\Delta}(\bb)\not=\emptyset.
\\
\end{aligned}
\label{e:Zygmund}
\end{equation}

\subsection{Poisson integrals and the basic Carleson tent estimate}

We denote by $|\Sb|$ the \textbf{Lebesgue measure} 
of a measurable subset $\Sb\subset\bhd$. 

A \textbf{null set} in $\bhd$ is a measurable subset 
$\Ub\subset\bhd$ with $|\Ub|=0$,
 and we set $\nullsets\eqdef\{\Ub\in\tps{\bhd}:
\Ub\text{ is a null set in } \bhd\}$. 
If $\Sb$ and $\Qb$ are subsets of $\bhd$ we say that 
$\Sb$ is \textbf{a.e. contained in} $\Qb$, and write
${\aesubset{\Sb}{\Qb}}$
if $\Sb\setminus\Qb\in\nullsets$: 
This means that almost all of $\Sb$ is a subset of $\Qb$. Observe that 
the sets $\Sb$ and $\Qb$ need not be measurable
and that 
\begin{equation}
\aesubset{\Ab\setminus\Bb}{\Cb}
\iff
\aesubset{\Ab\setminus\Cb}{\Bb}.
\label{e:aesubsetNEW} 
\end{equation}
If $\aesubset{\Sb}{\Qb}$ and $\aesubset{\Qb}{\Sb}$,  
we say that $\Sb$ and $\Qb$ are \textbf{almost everywhere equal}, 
and write 
${\aeequal{\Sb}{\Qb}}$.
A set $\Sb\subset\bhd$ 
has 
{\bf full measure} 
if $\aeequal{\Sb}{\bhd}$. 
A property is said to hold 
\textbf{a.e.}  
if the set of points 
in $\bhd$ 
for which it holds has full measure.

If $\bfunction:\bdryhd\to\RR$ is Lebesgue integrable,
we denote by $\PoissonI{\bfunction}:\hd\to\RR$ its 
\textbf{Poisson integral}; cf. \cite{SteinWeiss1971}. 

If $\Sb\subset\bdryhd$ then 
$\cfunction{\Sb}:\bdryhd\to\RR$ 
is defined by 
$\cfunction{\Sb}(\bpoint)=1$
if 
$\bpoint\in\Sb$,
$\cfunction{\Sb}(\bpoint)=0$
if 
$\bpoint\not\in\Sb$.

The following result is well-known (cf.\ \cite{Calderon1950}, 
\cite{SteinWeiss1971}). 
The constant $c_0>0$ that appears therein is called 
the \textbf{Carleson tent constant}. 
\begin{lemma} 
There exists a constant $c_0\in(0,+\infty)$ such that 
for each boundary ball $\bb\subset\bdryhd$ and each $\dpoint\in\boldsymbol{\Delta}(\bb)$, 
$\PoissonI{
\cfunction{\bb}}
(\dpoint)\geq{}c_0$.  
\label{l:CarlesonTent}
\end{lemma}
%\begin{proof}
%See e.g. Lemma~3.6 in \cite{DiBiaseStokolosSvenssonWeiss2006}. 
%\end{proof}

\subsection{The limsup lemma}
\label{s:limsup}

\begin{lemma}
If $\Vb \subset \bdryhd$ is open then, for all $\bpoint \in \Zb(\Vb)$,
$$
\limsup_{\stackrel{\dpoint\to\bpoint}{\dpoint\in\af(\bpoint)}}  
\PoissonI{\cfunction{\Vb}}(\dpoint)\geq c_0 
$$
\label{l:Carleson}
\end{lemma}
\begin{proof}
It follows from Lemma~\ref{l:CarlesonTent}, since if $\bb \subset \Vb$ then 
$\PoissonI{\cfunction{\bb}}\leq\PoissonI{\cfunction{\Vb}}$.
\end{proof}

\subsection{The liminf lemma}
\label{s:liminf}

\begin{lemma}
If $\Vb \subset \bhd$ is open and 
$|\bhd \setminus \Vb| > 0$, then 
\begin{equation}
\aesubset{\bhd \setminus \Vb}{
\left\{\bpoint\in\bhd:\liminf_{\stackrel{\dpoint\to\bpoint}{\dpoint\in\af(\bpoint)}} 
(\PoissonI{\cfunction{\Vb}})(\dpoint) =0\right\}}
\label{e:C1950}
\end{equation}
\label{l:C1950}
\end{lemma}
\begin{proof}
The proof will be given in Section~\ref{s:proofoflemmaC1950} 
\end{proof}

\subsection{Basic constructions}
\label{s:bs}

We may assume,
without loss of generality, that  
%\begin{equation}
%\varpaaw{\angolo}\subset\varpaaw{\angolo+1}
%\label{e:angolo} 
%\end{equation}
\begin{equation}
\angolo(\bpoint)\geq 10\quad\forall\bpoint\in\bhd
\label{e:largeangle} 
\end{equation}
Define 
$\tau : \bhd \times \hd \to (0, 1]$ 
by 
$\tau(\bpoint,\dpoint) \eqdef
\frac{\dpoint(\dimn)}{|\bpoint-\dpoint|}$, 
for 
$\bpoint\in\bhd$, 
$\dpoint\in\hd$, and 
consider the sequence of everywhere defined functions 
$\bfun_j: \bhd \to (0, \infty)$ gauging
the order of tangency of $\af(\bpoint)$ at various points: 
\begin{equation}
\bfun_j(\bpoint) \eqdef \sup\left\{
\tau(\bpoint, \dpoint) :
\dpoint\in\ballrel{\bpoint}{
\frac{3\sqrt{\dprime}}{4j}
}
\right\} 
\label{e:fn}
\end{equation}
Since, for each $\bpoint\in\bhd$, $\af(\bpoint)$ ends tangentially at $\bpoint$, 
the sequence 
${\{\bfun_n(\bpoint)\}}_{n\in\NN}$ decreases to $0$. Moreover, since 
$\af$ is regular, for 
each $n\in\NN$ the function $\bfun_n:\bhd\to(0,1]$ is measurable. 

Define the sequence $v:\NN\to\NN$ whose values are, in the 
natural order:
$$
2,3,2,3,4,2,
3,4,5,2,3,4,5,6,2,3,4,5,6,7, \ldots
$$
Hence for each $\bsegnato\in\NN\setminus\{1\}$, 
$v_j=\bsegnato$ for infinitely many values of $j\in\NN$.

The following observation will be helpful in the applications 
of Egorov's Theorem that will be given momentarily: 
If the sequence $\{\bfun_n\}$ converges uniformly to $0$ on a set $\Ab$ and on a set $\Bb$, then it
converges uniformly to $0$ on 
 $\Ab\cup \Bb$. 
In order to apply 
 Egorov's Theorem \cite{Royden1968bis}, we set 
 $\bb_n\eqdef\acal{{0}}{n}\subset\bhd$ where ${0}\in\bhd$ is the origin 
in $\RR^{\dimn}$, and define $\Xb_1\eqdef\bb_1$, $\Xb_k\eqdef\bb_k\setminus\bb_{k-1}$ if 
$k\geq2$. We now apply Egorov's Theorem to the sequence 
${\left\{\bfun_n\right\}}_n$ on each set $\Xb_k$, 
for each 
 $k\in\NN$, 
 and obtain a sequence ${\{\Sb(i,k)\}}_{i\in\NN}$ of subsets of 
 $\bhd$ such that $\forall k,i\in\NN$
\begin{description}
\item[(i)] 
the sequence ${\left\{\bfun_n\right\}}_n$ converges uniformly to 0 on $\Sb(i,k)$;

\item[(ii)] 
$\Sb(i,k)\subset\Sb(i+1,k)\subset\Xb_k$;

\item[(iii)]  
$|\Xb_k\setminus\Sb(i,k)|<2^{-i}|\Xb_k|$. 

\end{description}
If $\Cb_i\eqdef\bigcup_{j=1}^{i}\Sb(i,j)$, then 
 the sequence ${\left\{\bfun_n\right\}}_n$ converges uniformly to 0 on $\Cb_i$, 
\begin{equation}
\Cb_j \subset \Cb_{j+1}
\quad\forall\, j\in\NN
\label{e:Cisincreasing} 
\end{equation}
and    
 \begin{equation}
\Cb=\bigcup_{i=1}^{\infty}\Cb_i=\bigcup_{i=1}^{\infty}\bigcup_{k=1}^{i}\Sb(i,k)
\stackrel{\text{(a)}}{=}
\bigcup_{i=1}^{\infty}\bigcup_{k=1}^{\infty}\Sb(i,k)
\label{e:dofC} 
\end{equation}
where (a) follows from (ii).  Thus $\forall i,k$, 
$\Xb_k\setminus\Cb\subset\Xb_k\setminus\Sb(i,k)$, and 
hence (iii) implies that $|\Xb_k\setminus\Cb|=0$. Since $\bhd=\bigcup_{k\in\NN}\Xb_k$, 
$\Cb$ has full measure in $\bhd$.  
It follows that there exists a sequence 
$\varphi:\NN\to\NN$ such that, for each $j\in\NN$ 
\begin{equation}
\begin{aligned}
& 
\textup{(i)}
\,\,
\sup \left\{\bfun_{{\varphi_j}}(\vb): \vb \in \Cb_j\right\} 
< c\frac{10}{22}\frac{2^{-j}}{v_j},\quad
\textup{(ii)}
\,\,
\varphi_j>j,
\quad
\textup{(iii)}\,
\varphi_{j+1}>\varphi_j 
\\
&
\text{ (iv) there exists }q_j\in\NN
\text{ such that } q_j/{\varphi_j}=j 
\end{aligned}
\label{e:aenne} 
\end{equation}
where $c$ is a constant that will be chosen momentarily. 

The \textit{cube in} $\bhd$ 
centered at $\bpoint\in\bhd$ of side length $\rho>0$ is defined as follows:
$$
\acube{\bpoint}{\rho}=
\left\{
\bpointv\in\bhd:
\forall i=1,2,\ldots,\dprime, 
|\bpointv(i)-\bpoint(i)|<\frac{1}{2}\rho
\right\}
$$
Observe that $|\acube{\bpoint}{\rho}|=\rho^{\dprime}$. 

If $j\in\NN$ and $j\geq1$, we define $\Ub_j$ as the set of points $\bpoint\in\bhd$ such that for each $i\in\{1,2,\ldots,\dprime\}$ there exists 
$k_i\in\{0,\pm1,\pm2,\ldots,\pm q_j\}$ such that 
$\bpoint(i)=k_i\cdot{}{(\varphi_j)}^{-1}$. 
Observe that
\begin{equation}
\aeequal{\bigcup_{\bpoint\in\Ub_j}\acube{\bpoint}{{(\varphi_j)}^{-1}}}{
\left\{
\bpoint\in\bhd:\forall i=1,2,\ldots,{\dprime}, |\bpoint(i)| < j+\frac{1}{\varphi_j}
\right\}}
\label{e:definitionofy}
\end{equation}
and the cubes $\acube{\bpoint}{{(\varphi_j)}^{-1}}$, which appear in the union in the 
left-hand side of~\eqref{e:definitionofy}, are disjoint.  
We denote by $\Sb_j$ the right-hand side of~\eqref{e:definitionofy} and define 
\begin{equation}
\Ob_j  \eqdef 
\bigcup_{\bpoint\in\Ub_{j}}\acube{\bpoint}{2^{-j}\varphi_j^{-1}}
\label{e:definitionofOn}
 \end{equation}
and
\begin{equation}
\Vb_n\eqdef\bigcup_{j\geq{}n}\Ob_j
\label{e:Vn} 
\end{equation}
Since $|\Ob_j|={(2^{-j})}^{\dprime}|\Sb_{j}|={(2^{-j})}^{\dprime}{{[2(j+\varphi_j^{-1})]}^{\dprime}}$, 
the set $\Vb_n\subset\bhd$ has small measure, because 
$$
\sum_{j}{(2^{-j})}^{\dprime}{{[2(j+\varphi_j^{-1})]}^{\dprime}}<+\infty.
$$
Observe that $\Vb_n$
is open in $\bhd$ and dense in $\bhd$. 

\subsection{The geometric lemma}
\label{s:basiclemmata}

\begin{lemma} For each $n\in \mathbb{ N}$,
$\Cb \setminus \Vb_n \subset \Zb(\Vb_n)$. 
\label{l:ip}
\end{lemma}
\begin{proof}
The proof will be given in Section~\ref{s:proofoflemmaip} 
\end{proof}

\subsection{Proof of Theorem~\ref{thm:maintheorem}}\label{s:proofofmainresult}

Here we adopt the same notation as in Section~\ref{s:bs}. 
Observe that, for each $j\in\NN$, $\PoissonI{\cfunction{\Vb_j}} \in {\hinfty}$. 
Then Lemma~\ref{l:Carleson}, 
Lemma~\ref{l:ip}, Lemma~\ref{l:C1950}, and 
imply that 
\begin{equation}
\text{for a.e.} 
\bpoint\in\Cb\setminus\Vb_j,
\osc{\PoissonI{\cfunction{\Vb_j}}}{\bpoint}{\af} \geq c_0
\label{e:aepo}
\end{equation}
where $c_0$ is the Carleson tent constant. 
Let $s \eqdef{}1+\frac{1+c_0}{c_0}$ and, as in  
\cite{Zygmund1949}, 
$$
\bfunction=\sum_{j\geq1}s^{-j}\cfunction{\Vb_j}
$$
It follows that 
$$
\PoissonI{\bfunction}=\sum_{j\geq1}s^{-j}\PoissonI{\cfunction{\Vb_j}}
$$
Define $\Wb\eqdef\cap_{j\geq1}\Vb_j$ and observe that 
$|\Wb|=0$.
\begin{claim}
For a.e. $\bpoint\in\Cb\setminus\Wb$, $\osc{\PoissonI{\bfunction}}{\bpoint}{\af}>0$.
\label{claim:1} 
\end{claim}

Since the set $\Cb\setminus\Wb$ has full measure in $\bhd$, Claim~\ref{claim:1} completes the proof of Theorem~\ref{thm:maintheorem}. 

\paragraph{Proof of Claim~\ref{claim:1}.}

Let $\bpoint\in\Cb\setminus\Wb$ and let $j$ be the smallest integer $n$ such that $\bpoint \notin \Vb_n$. 
Then $\bpoint$ belongs to the open set 
\begin{equation}
\label{1eqqqnn3}
\bigcap^{j-1}_{k=1} \Vb_k.
\end{equation}
For $k=1,2,\ldots, j-1$, the function $\cfunction{\Vb_k}$ is equal to 1 on the set \eqref{1eqqqnn3}; since this set is open, it follows that 
$\osc{\PoissonI{\cfunction{\Vb_k}}}{\bpoint}{\af}=0$
for each $k = 1, 2, \ldots, l-1$.  

Observe that \eqref{e:aepo} implies that 
$\osc{s^{-j}\PoissonI{\cfunction{\Vb_j}}
}{\bpoint}{\af}\geq s^{-j}c_0$ (discarding a null set $\Nb_j$). On the other hand, 
$$
\osc{\sum_{k\geq{}j+1}s^{-k}\PoissonI{(\cfunction{\Vb_k})}}{\bpoint}{\af}
\leq
\sum_{k\geq{j+1}}s^{-k}
\leq
s^{-j}\frac{1}{s-1}
$$
It follows that if $\bpoint\not\in\Nb$, where $\Nb\eqdef\bigcup_k\Nb_k$, 
$\osc{\PoissonI{\bfunction}}{\bpoint}{\af}\geq
s^{-j}c_0-s^{-j}\frac{1}{s-1}>0$.
Since the set 
$(\Cb \setminus \Wb) \setminus \Nb$ 
has full measure in 
$\Cb \setminus \Wb$, the proof of Claim~\ref{claim:1} is completed.

\subsection{Proof of Lemma~\ref{l:C1950}}
\label{s:proofoflemmaC1950}

If we denote $\bhd \setminus \Vb$ by $\wVb$ and 
the set on the right-hand side of~\eqref{e:C1950} by 
$\Qb$, then we have 
to prove that $\aesubset{\wVb}{\Qb}$, i.e., 
that $\wVb\setminus\Qb\in\nullsets$. 
This statement follows from the following claim. 
\begin{claim}
For each $\delta>0$ 
there exists  $\Sb\subset\wVb$ with 
$|\wVb\setminus\Sb|<\delta$ and 
\begin{equation}
\aesubset{\Sb}{\Qb}
\label{e:claim1} 
\end{equation}
\label{claim:oneNEW} 
\end{claim}

\paragraph{Proof of Claim~\ref{claim:oneNEW}.} 
Let $\delta>0$. 
A well-known extension of Fatou's Theorem to Poisson integrals 
(see e.g. \cite{SteinWeiss1971}) says that 
$$
\aeequal{\Fatouset{\PoissonI{
\cfunction{\Vb}
}}}{\bhd}
\text{ and }
\ntbv{(\PoissonI{\cfunction{\Vb}})}(\bpoint)=1_{\Vb}(\bpoint),  
\forall\bpoint\in\Fatouset{P(1_{\Vb})}.
$$ 
It follows that 
$\ntbv{(\PoissonI{
\cfunction{\Vb}})}(\bpoint)=0$
for a.e. 
$\bpoint\in\wVb$.

We denote by $\Xb$ the subset of $\wVb$ of full measure 
in $\wVb$ with the property that 
$\ntbv{(\PoissonI{\cfunction{\Vb}})}(\bpoint)=0$
for all $\bpoint\in\Xb$, and define, 
for each $(j,b)\in\NN\times\NN$, the function 
$$
\bfung_{(j,b)}:\Xb\to(0,1]
\text{ by }
\bfung_{(j,b)}(\bpoint)\eqdef\sup\{
\PoissonI{\cfunction{\Vb}}(\dpoint):\dpoint\in\Gamma_{b}(\bpoint)\cap\ball{\bpoint}{1/j}\}.
$$
Observe that, $\forall\bpoint\in\Xb$ and $\forall j,b\in\NN$
\begin{equation}
(i) \lim_{j\to+\infty}\bfung_{(j,b)}(\bpoint)=0,\,
(ii)\, \bfung_{(j+1,b)}(\bpoint)\leq\bfung_{(j,b)}(\bpoint)\leq \bfung_{(j,b+1)}(\bpoint) 
\label{e:bfung}
\end{equation}
and define 
\begin{equation}
\bfunp_j(\bpoint)\eqdef\sum_{b=1}^{+\infty}2^{-b}\bfung_{(j,b)}(\bpoint)
\label{e:bfunp}
\end{equation} 
We claim that $\lim_{j\to+\infty}\bfunp_{j}(\bpoint)=0$ 
$\forall\bpoint\in\Xb$. 
Indeed, if $\bpoint\in\Xb$ and $\epsilon>0$ let $b_{\epsilon}\in\NN$ with
$\sum_{b>b_{\epsilon}}2^{-b}<\epsilon/2$, and let $j_{(\epsilon,\bpoint)}\in\NN$ such that 
if $j\geq{}j_{(\epsilon,\bpoint)}$ then $\bfung_{(j,b_\epsilon)}(\bpoint)<\frac{\epsilon2^{b_\epsilon}}{2(2^{b_\epsilon}-1)}$. 
If $j\geq{}j_{(\epsilon,\bpoint)}$, \eqref{e:bfung}-(ii) implies that 
$\bfunp_j(\bpoint)$ is bounded by 
$\sum_{b=1}^{b_\epsilon}\bfung_{(j,b)}(\bpoint)2^{-b}
+
\sum_{b>b_\epsilon}\bfung_{(j,b)}(\bpoint)2^{-b}
\leq
\bfung_{(j,b_\epsilon)}(\bpoint)\sum_{b=1}^{b_\epsilon}2^{-b}
+
\sum_{b>b_\epsilon}2^{-b}<\epsilon$.

Write $\Xb$ as the union $\Xb=\bigcup_{k\in\NN}{\Xb_k}$, where 
the sets $\Xb_k$ are disjoint,  
each set $\Xb_k$ is measurable, and it has positive measure bounded by $1$. 
For each $k\in\NN$, apply 
Egorov's theorem to $\Xb_k$ with respect to the sequence 
${\{\bfunp_{j}\}}_{j\in\NN}$. 
Then for each $i\in\NN$ there exists a perfect subset 
$\Sb(i,k)\subset\Xb_k$ such that the sequence 
${\{\bfunp_{j}\}}_{j\in\NN}$
converges uniformly to zero on 
$\Sb(i,k)$ and 
$|\Xb_{k}\setminus\Sb(i,k)|<\delta2^{-i}|\Xb_k|$. We may assume, without loss of generality, that $\Sb(i,k)\subset\Sb(i+1,k)$. 

Define, for each $b\in\NN$, $\Sb_b\eqdef\bigcup_{i=1}^{b}\Sb(b,i)$: Then $\Sb_b$ is perfect, 
\begin{equation}
\Sb_b\subset\Sb_{b+1}\subset\Xb
\label{e:increasing} 
\end{equation}
and, for each $b\in\NN$, the sequence ${\{\bfunp_{j}\}}_{j\in\NN}$
converges uniformly to zero on $\Sb_b$. 

Let $\Sb\eqdef\bigcup_{b=1}^{+\infty}\Sb_b\subset\Xb$. Then $|\Xb\setminus\Sb|<\delta$, since 
(a) 
$\Xb\setminus\Sb=\bigcup_{b=1}^{+\infty}\Xb_b\setminus\Sb$; 
for each $b\in\NN$, $\Xb_b\setminus\Sb\subset\Xb_b\setminus\Sb(b,b)$, 
and hence 
$$
|\Xb\setminus\Sb|<\sum_{b=1}^{+\infty}|\Xb_b|2^{-b}\delta
\leq
\sum_{b=1}^{+\infty}2^{-b}\delta=\delta
$$

Let $\bopoint\in\Sb$ and assume that 
$\af(\bopoint)$ is $b_1$-firm at $\bopoint$, $b_1\in\NN$.
Assume that $\bopoint\in\Sb_{b_2}$, $b_2\in\NN$, and 
let $b_0\eqdef\max\{b_1,b_2\}$. Then 
$\af(\bopoint)$ is $b_0$-firm at $\bopoint$ and~\eqref{e:increasing} implies that  
$\bopoint\in\Sb_{b_0}$. 

Let $\epsilon>0$. Since the sequence 
${\{\bfunp_{j}\}}_{j\in\NN}$ converges to zero uniformly on $\Sb_{b_0}$, 
there exists $n_{\epsilon}\in\NN$ such that 
\begin{equation}
\text{if }
j\geq n_{\epsilon}
\text{ then }
\bfunp_{j}(\bpointx)
<2^{-b_0}\epsilon \text{ for all } \bpointx\in\Sb_{b_0}
\label{e:supUNO} 
\end{equation}
Let $r\in(0,\frac{2^{-10}}{n_0})$. 
Since $\af(\bopoint)$ is $b_0$-firm at $\bopoint$,
there exists $s_r>0$ such that 
\begin{equation}
\acal{\bopoint}{s_r}\subset\Gamma_{b_0}^{\ast}(\ballrel{\bopoint}{r})
\label{e:shadow2}
\end{equation}
We may assume, without loss of generality, that $s_r<{2^{-10}}{r}$.

Since $\Sb_{b_0}$ is perfect, there exists 
$\bpoint_r\in\Sb_{b_0}\cap(\acal{\bopoint}{s_r}\setminus\{\bopoint\})$. 
Then~\eqref{e:shadow2} implies that 
$$
\bpoint_r\in\Gamma_{b_0}^*(\{\dpoint_r\}) 
\text{ for some }
\dpoint_r\in\ballrel{\bopoint}{r}
$$
hence $\dpoint_r\in\Gamma_{b_0}(\bpoint_r)\cap\af(\bopoint)\cap\ball{\bopoint}{r}$. 
Observe that 
$$
|\dpoint_r-\bpoint_r|\leq
|\dpoint_r-\bopoint|+|\bopoint-\bpoint_r|<r+s_r
<\frac{2^{-10}}{n_{\epsilon}}+\frac{2^{-20}}{n_{\epsilon}}
<\frac{1}{n_{\epsilon}}
$$
Hence $\dpoint_r\in\ball{\bpoint_r}{\frac{1}{n_{\epsilon}}}
\cap{}\Gamma_{b_0}(\bpoint_r)$ and 
$\bpoint_{r}\in\Sb_{b_0}$.  
Then~\eqref{e:bfunp} and~\eqref{e:supUNO} yield
$$
2^{-b_0}\bfung_{(n_\epsilon,b_0)}(\bpoint_r)\leq\bfunp_{n_\epsilon}(\bpoint_r)<
2^{-b_0}\epsilon
$$
(with $\bpointx=\bpoint_r$) and thus 
$$
(\PoissonI{\cfunction{\Vb}})(\dpoint_r)\leq{}g_{(n_0,b_0)}(\bpoint_r)<\epsilon.
$$
Hence we have proved that $\forall\bopoint\in\Sb$ and $\forall\epsilon>0$ 
$\exists\, n_{\epsilon}\in\NN$ such that $\forall \, r\in(0,\frac{2^{-10}}{n_{\epsilon}})$ 
there exists $\dpoint_r\in\af(\bopoint)\cap\ball{\bopoint}{r}$
with 
$(\PoissonI{1_{\Vb}})(\dpoint_r)<\epsilon$, and this implies that 
\begin{equation}
\liminf_{\stackrel{\dpoint\to\bopoint}{\dpoint\in\af(\bopoint)}}\, 
(\PoissonI{\cfunction{\Vb}})(\dpoint) =0.
\label{e:almost}
\end{equation}
Hence $\Sb\subset\Qb$. 
We have thus proved that $\forall\delta>0$ there exists 
$\Sb\subset\Xb$ with $|\Xb\setminus\Sb|<\delta$ and $\Sb\subset\Qb$. It follows that 
$\aesubset{\Xb}{\Qb}$, and since $\Xb$ has full measure in $\wVb$, it follows that 
$\aesubset{\wVb}{\Qb}$. The proof is concluded.

\subsection{Proof of Lemma~\ref{l:ip}}
\label{s:proofoflemmaip}

Before we delve into the proof, we would like to explain why it is more 
involved than one would perhaps think. 
It is indeed true that, given $\bpoint,\bpointv\in\bhd$, with 
$\bpoint\not=\bpointv$, and given $b\in\NN$ and $r>0$ there exists
$n=n(\bpoint,\bpointv,b,r)\in\NN$ such that 
$\Gamma_{b}(\bpointv)\setminus\Gamma_{n}(\bpoint)\subset\ball{\bpointv}{r}$: 
However, in this statement, $n$ depends upon $\bpoint$, 
$\bpointv$, 
$b$,
and $r$, 
while, in the statement that we need to prove, the variables appear in a different order 
and hence the dependence of $n$ on the other variables is not admissible. 

Let $\bopoint\in\Cb\setminus\Vb_n$. 
Observe that there exists $\bsegnato\in\NN$ such that 
\begin{equation}
\text{(i) } \bsegnato\geq 10
\text{ and (ii) }
 \af(\bopoint)
\text{ is $\bsegnato$-firm at } \bopoint 
\label{e:bsegnato}
\end{equation}
Since $\bopoint\in\Cb\setminus\Vb_n$, there exists $\dnewsegnato\in\NN$ with 
$\bopoint\in\Cb_{\dnewsegnato}$, and \eqref{e:Cisincreasing} implies  that 
\begin{equation}
 j\geq \dnewsegnato\implies \bopoint\in\Cb_j
\text{ and hence }
\eqref{e:aenne}
\text{ implies that }
\bfunction_{\varphi_j}(\bopoint)\leq{}c\frac{10}{22}\frac{2^{-j}}{v_j}
\label{e:fromd}
\end{equation}
Moreover, 
\begin{equation}
\forall{}j\geq n, \,\,\bopoint\not\in \Ob_j
\label{e:fromdBIS}
\end{equation}
Let $\epsilon>0$. 
Since $\af(\bopoint)$ is tangential to $\bhd$ at $\bopoint$
there exists $r=r_{\epsilon}>0$ such that 
\begin{equation}
\sup\{\tau(\bopoint,\dpoint):\dpoint\in
\ballrel{\bopoint}{r_{\epsilon}}
\}<\frac{1}{10}\frac{1}{{(1+\bsegnato)}^2}
\text{ and }
r_{\epsilon}<\epsilon\cdot{}2^{-10}
\label{e:epsilon'} 
\end{equation}
Then \eqref{e:bsegnato} implies that   
\begin{equation}
\Gamma_{\bsegnato}^{\ast}(\ballrel{\bopoint}{r_{\epsilon}})
\text{ contains a boundary ball of center }
\bopoint
\label{e:padj} 
\end{equation}
i.e., there exists $\newradius=\newradius_{r}=\newradius_{r_{\epsilon}}>0$ such that 
\begin{equation}
\textup{(i) }
\,
\acal{\bopoint}{\newradius}
\subset\Gamma_{\bsegnato}^{\ast}(\ballrel{\bopoint}{r_{\epsilon}})
\label{e:better}
\end{equation}
Moreover, we may assume, without loss of generality, that
\begin{equation}
\textup{ (ii) }
\,
\newradius<r_{\epsilon}\cdot{}2^{-10}.
\label{e:theta} 
\end{equation}
Now, select $j=j_{\epsilon}\in\NN$ such that 
\begin{equation}
\begin{aligned}
&\textup{ (i) } 
j_{\epsilon}>\max\{n,10\}
\quad 
\textup{ (ii) } 
\frac{1}{\varphi_{j_{\epsilon}}}<\frac{2}{\sqrt{\dprime}}\newradius\cdot{2^{-10}}
\quad 
\textup{ (iii) }
j_{\epsilon}>\dnewsegnato
\\
&\textup{ (iv) } v_{j_{\epsilon}}=1+\bsegnato
\quad 
\textup{ (v) }  
j_{\epsilon}>2\max\{\bopoint(1),\bopoint(2),\ldots,\bopoint(\dprime)\}
\end{aligned}
\label{e:selectenne}
\end{equation}
Let $\vob\in\Ub_j$ 
a point of $\Ub_j$ that minimizes the distance from $\bopoint$. Then 
\begin{equation}
2^{-j}\frac{1}{\varphi_j}\stackrel{(a)}{\leq}|\vob-\bopoint|\stackrel{(b)}{\leq}\frac{1}{2}\sqrt{\dprime}\frac{1}{\varphi_{j}}
\label{e:estimate1NEW} 
\end{equation}
where (a) follows from the fact that $\bopoint\not\in\Ob_j$
and (b) from the fact that $\vob$ the point of $\Ub_j$ closest to $\bopoint$ 
and hence $|\vob-\bopoint|$ cannot exceed half the length of the diagonal 
of the cube $\acube{\vob}{{(\varphi_j)}^{-1}}$. 

Observe that 
$$
\vob\in\acal{\bopoint}{s}\stackrel{(c)}{\subset}\acal{\bopoint}{r_{\epsilon}}
\stackrel{(d)}{\subset}\Gamma_{\bsegnato}^{\ast}(\ballrel{\bopoint}{r_{\epsilon}})
$$
where (c) follows from~\eqref{e:theta} and
(d) follows from~\eqref{e:better}, hence  
$\vob\in\Gamma_{\bsegnato}^{\ast}(\ballrel{\bopoint}{r_{\epsilon}})$, 
and therefore 
\begin{equation}
\text{there exists }
\zsegnato\in\ballrel{\bopoint}{r_{\epsilon}}
\text{ such that }
\zsegnato\in\Gamma_{\bsegnato}(\vob)
\label{e:projectionNEW} 
\end{equation}
Since 
$\zsegnato\in\ballrel{\bopoint}{r_{\epsilon}}$, \eqref{e:epsilon'} implies that 
$\tau(\bopoint,\zsegnato)<\frac{1}{10}\frac{1}{{(1+\bsegnato)}^2}$, i.e.
\begin{equation}
\zsegnato(\dimn)<|\bopoint-\zsegnato|\frac{1}{10}\frac{1}{{(1+\bsegnato)}^2}
\label{e:alzataNEW} 
\end{equation}
Since $\zsegnato\in\Gamma_{\bsegnato}(\vob)$ it follows that 
$\tau(\vob,\zsegnato)>1/(1+\bsegnato)$, i.e. 
\begin{equation}
|\vob-\zsegnato|<\zsegnato(\dimn) (1+\bsegnato)
\label{e:alzata1NEW}
\end{equation}
Then~\eqref{e:alzataNEW} and~\eqref{e:alzata1NEW} imply that 
\begin{equation}
|\vob-\zsegnato|<|\bopoint-\zsegnato| \frac{1}{10}\frac{1}{1+\bsegnato}
\label{e:alzata2}
\end{equation}
Then 
$|\bopoint-\zsegnato|\leq|\bopoint-\vob|+|\vob-\zsegnato|<|\bopoint-\vob|+|\bopoint-\zsegnato| 
\frac{1}{10}\frac{1}{1+\bsegnato}$, hence
\begin{equation}
|\bopoint-\zsegnato|\left[1-\frac{1}{10}\frac{1}{1+\bsegnato}\right]<|\bopoint-\vob|
\label{e:alzata3NEW}
\end{equation}
Observe that
$
\displaystyle{1-\frac{1}{10}\frac{1}{1+\bsegnato}=\frac{1+\frac{9}{10\bsegnato}}{1+\frac{1}{\bsegnato}}}$ and hence~\eqref{e:estimate1NEW}-(b) 
and~\eqref{e:alzata3NEW} imply 
\begin{equation}
|\bopoint-\zsegnato|<
\frac{1+\frac{1}{\bsegnato}}{1+\frac{9}{10\bsegnato}}\frac{1}{2}\sqrt{\dprime}\frac{1}{\varphi_{j}}
\stackrel{(\text{e})}{<}\frac{3}{4}\sqrt{\dprime}\frac{1}{\varphi_{j}}
\label{e:alzata5NEW}
\end{equation}
where (e) follows from~\eqref{e:bsegnato}-(i), which implies that 
$\displaystyle{\frac{1+\frac{1}{\bsegnato}}{1+\frac{9}{10\bsegnato}}\leq1+\frac{1}{109}}$. 
Now observe that~\eqref{e:projectionNEW} and~\eqref{e:alzata5NEW} imply that 
$\zsegnato\in\ballrel{\bopoint}{\frac{3\sqrt{\dprime}}{4\varphi_{j}}}$, 
and thus~\eqref{e:fn} implies that  
\begin{equation}
\tau(\bopoint,\zsegnato)\leq{}f_{\phi_j}(\bopoint)
\label{e:alzata7NEW}
\end{equation} 
Since 
$\bopoint\in\Cb_j$, \eqref{e:aenne} 
and~\eqref{e:selectenne}
imply that 
\begin{equation}
f_{\phi_j}(\bopoint)<c\frac{10}{22}\frac{2^{-j}}{v_j}
=
c\frac{10}{22}\frac{2^{-j}}{1+\bsegnato}
\label{e:alzata8NEW}
\end{equation} 
\eqref{e:alzata7NEW} and~\eqref{e:alzata8NEW} imply that 
$\tau(\bopoint,\zsegnato)=\frac{\zsegnato(\dimn)}{|\bopoint-\zsegnato|}
<c\frac{10}{22}\frac{2^{-j}}{1+\bsegnato}$, i.e., 
\begin{equation}
\zsegnato(\dimn)<c
|\bopoint-\zsegnato|
\frac{10}{22}\frac{2^{-j}}{1+\bsegnato}
\label{e:schiacciataNEW}
\end{equation} 
and now~\eqref{e:alzata1NEW}, 
\eqref{e:schiacciataNEW}, and \eqref{e:alzata5NEW}
imply that
\begin{equation}
|\vob-\zsegnato|<c\frac{3}{4}\sqrt{\dprime}\frac{1}{\varphi_{j}}\frac{10}{22}2^{-j}
\label{e:schiacciatabisNEW}
\end{equation} 
Let $c=\frac{22}{10\sqrt{\dprime}}$. Then~\eqref{e:schiacciatabisNEW} implies 
that $|\vob-\zsegnato|<\frac{3}{4}2^{-j}\frac{1}{\varphi_j}$.  
and hence 
\begin{equation}
\zsegnato\in \boldsymbol{\Delta}(\acal{\vob}{2^{-j}\frac{1}{\phi_j}})
\label{e:schiacciatatentNEW}
\end{equation} 
Now observe that 
$\acal{\vob}{2^{-j}\frac{1}{\phi_j}}\subset\Ob_j\subset\Vb_n$ and that 
$\acal{\vob}{2^{-j}\frac{1}{\phi_j}}\subset\ball{\bopoint}{\epsilon}$. 
These facts conclude the proof.

\paragraph{Tool and computational resource disclosure.}

The research that led to this paper was conducted and this paper was written 
without recourse to AI. 

\paragraph{Acknowledgement.}
This research has been supported by the Sida-funded UR-Sweden Program for Research, Higher Learning
and Institution Advancement, sub-program Strengthening Research Capacity in Mathematics, Statistics
and Their Applications.
Partial support from Fondi di Ricerca di Ateneo 
of the Università “G. d'Annunzio” in Chieti and Pescara, 
and from 
Indam-Gnampa, in Italy,  is gratefully acknowledged.

\end{document}